\documentclass{amsart}
\usepackage{graphicx}
\usepackage{color}
\usepackage{amsfonts,amssymb,amscd,amsmath,enumerate,verbatim,newlfont,calc}

\usepackage{amsfonts}

\newtheorem{theorem}{Theorem}[section]

\theoremstyle{definition}
\newtheorem{definition}[theorem]{Definition}
\newtheorem{example}[theorem]{Example}

\theoremstyle{remark}

\numberwithin{equation}{section}
\begin{document}

\title[HYPERBOLIC DYNAMICS OF DISCRETE DYNAMICAL SYSTEMS ON ... ]{HYPERBOLIC DYNAMICS OF DISCRETE DYNAMICAL SYSTEMS ON PSEUDO-RIEMANNIAN MANIFOLDS }
\author[MohammadReza Molaei]{ {\bf MohammadReza Molaei
}\\ {Mahani Mathematical Research Center\\ Shahid Bahonar University of Kerman, Kerman, Iran\\e-mail: mrmolaei@uk.ac.ir }}
\maketitle

\begin{abstract}
We consider a discrete dynamical system on a pseudo-Riemannian manifold and we determine the concept of a hyperbolic set for it. We insert a condition in the definition of a hyperbolic set which implies to the unique decomposition of a part of tangent space (at each point of this set) to two unstable and stable subspaces with exponentially increasing and exponentially decreasing dynamics on them. We prove the continuity of this decomposition via the metric created by a torsion-free pseudo-Riemannian connection. We present a global attractor for a diffeomorphism on an open submanifold of the hyperbolic space $H^{2}(1)$ which is not a hyperbolic set for it.
\end{abstract}
\section{Introduction}
Hyperbolic dynamics on a Riemannian manifold has a deep history in mathematics and physics \cite{A,H,P}, and it is a main mathematical tool to determine the complex systems behavior \cite{G,R,Y}.   The phrase "{\it hyperbolic dynamics}" has different definitions  in partial differential equations, ordinary differential equations, and discrete dynamical systems \cite{AV,B,C,HA,PA,PT,S}. The notion of hyperbolic dynamics for partial differential equations in a pseudo-Riemannian manifold has been considered by Choquet-Bruhat, and Ruggeri via considering  hyperbolicity of the 3+1 system of Einstein equations \cite{C} in 1983. Here we extend {\it the notion of hyperbolic sets} to discrete dynamical systems on pseudo-Riemannian manifolds. The appearance of a non-trivial hyperbolic set for discrete dynamical systems on  Riemannian manifolds return to Smale's paper in 1998 by introducing  Smale's horseshoe \cite{S}.\\We assume that $M $ is a
 pseudo-Riemannian manifold with the pseudo-Riemannian metric $g(.,.)$. If $p\in M$, then a vector $v\in T_{p}M$ is called spacelike, timelike or null if $g_{p}(v,v)>0$, $g_{p}(v,v)<0$, or $g_{p}(v,v)=0$ respectively. In the next section we present the definition of a hyperbolic set for a discrete dynamical system created by a diffeomorphism on $M$. We see that the derivative of a hyperbolic dynamics affect on two sets of  non-null vectors, and the iteration of it or in it's inverse creates an exponentially growth on the length of vectors. A distribution of null vectors has essential role in this kind of dynamics. In theorem 2.1 we prove the unique decomposition of a part of tangent space at a point of a hyperbolic set to stable and unstable subspaces up to a  special distribution. We prove the continuity of this decomposition in theorem 2.2. In section 3 we present examples of hyperbolic sets and in example 4.1 we find a global attractor which is not hyperbolic.
\section{Hyperbolic set}
In this section we assume that $(M,g)$ is a smooth pseudo-Riemannian manifold, and $f:M\rightarrow M$ is a diffeomorphism.  We also assume that $C $ is a compact invariant set for $f$ i.e. $f^{-1}(C)=C$.
\begin{definition}\label{d1}
We say that $C$ is a hyperbolic set for $f$ up to a distribution $p\mapsto E^{n}(p)$, if there exist positive constants $a$ and $b$ with $b<1$ and a decomposition $$T_{p}M=E^{s}(p)\oplus E^{u}(p)\oplus E^{n}(p)$$
for each $p\in C$ so that:\\
$(i)$ Each non-zero vector in $E^{s}(p)$ or $E^{u}(p)$ is timelike or spacelike, and each vector of $E^{n}(p)$ is a null vector;\\
$(ii)$ $Df_{p}E^{s}(p)=E^{s}(f(p))$ and $Df_{p}E^{u}(p)=E^{u}(f(p))$;\\
$(iii)$ if $v\in E^{s}(p)$ and $n\in N$ then  $|g_{f^{n}(p)}(Df^{n}_{p} (v), Df^{n}_{p} (v))|\leq a b^{n}|g_{p}(v,v)|$ and $$\lim_{n\rightarrow \infty }g_{f^{n}(p)}(Df^{n}_{p} (v),Df^{n}_{p}(w))=0 ~ for~all~w\in T_{p}M ~with~the~following~property$$ $$| g_{f^{n}(p)}(Df^{n}_{p} (w), Df^{n}_{p} (w))|\leq a b^{n}|g_{p}(w,w)|~for ~all~n\in N;$$\\
$(iv)$ if $v\in E^{u}(p)$ and $n\in N$ then $a^{-1}b^{-n}|g_{p}(v,v)|\leq |g_{f^{n}(p)}(Df^{n}_{p} (v), Df^{n}_{p} (v))|.$
\end{definition}
\begin{theorem}\label{t1}
Let $C$ be a hyperbolic set for $f$ up to a distribution $p\mapsto E^{n}(p)$. Then for each $p\in C$, the tangent space of $M$ at $p$ has a unique decomposition with the properties of the former definition.
\end{theorem}
\begin{proof}
Let $p\in C$, and let $T_{p}M=E^{s}_{1}(p)\oplus E^{u}_{1}(p)\oplus E^{n}(p)=E^{s}_{2}(p)\oplus E^{u}_{2}(p)\oplus E^{n}(p)$, where $E^{s}_{i}(.)$, and $E^{s}_{i}(.)$ satisfy the properties of definition \ref{d1}. Then $E^{s}_{1}(p)\oplus E^{u}_{1}(p)=E^{s}_{2}(p)\oplus E^{u}_{2}(p).$
If $u\in E^{s}_{1}(p)$, then $u=v+w$, where $v\in E^{s}_{2}(p)$ and $w\in E^{u}_{2}(p)$. For $n\in N$ we have $$a^{-1}b^{-n}|g_{p}(w,w)|\leq |g_{f^{n}(p)}(Df^{n}_{p}(w), Df^{n}_{p}(w))|$$ $$=|g_{f^{n}(p)}(Df^{n}_{p}(u-v), Df^{n}_{p}(u-v))|$$ $$= |g_{f^{n}(p)}(Df^{n}_{p}(u), Df^{n}_{p}(u))+ g_{f^{n}(p)}(Df^{n}_{p}(v), Df^{n}_{p}(v))- 2g_{f^{n}(p)}(Df^{n}_{p}(u), Df^{n}_{p}(v))|$$ $$\leq |g_{f^{n}(p)}(Df^{n}_{p}(u), Df^{n}_{p}(u))|+ |g_{f^{n}(p)}(Df^{n}_{p}(v), Df^{n}_{p}(v))|+ 2|g_{f^{n}(p)}(Df^{n}_{p}(u), Df^{n}_{p}(v))|.$$
Since the right hand side of the above inequality tends to zero when $n$ tends to infinity, then $|g_{p}(w,w)|=0$. Thus $w\in E^{n}(p)\cap E^{u}_{2}(p)=\{0\}$. Hence $E^{s}_{1}(p)\subseteq E^{s}_{2}(p)$. The similar calculations imply that $E^{s}_{2}(p)\subseteq E^{s}_{1}(p)$, so they are equal, and we have a unique decomposition.
\end{proof}
Now we are going to define a metric on the set of subspaces of tangent space at a point of $M$. For this purpose we use of "parallel translation". Let us recall it. We assume that $\mathfrak{X}(M)$ is the set of smooth vector fields on $M$. A mapping $\nabla : \mathfrak{X}(M)\times \mathfrak{X}(M)\rightarrow \mathfrak{X}(M) $ ($(X,Z)\mapsto \nabla _{X} ^{Z})$) is called an affine connection on $M$ \cite{a1} if it satisfies the following three conditions:\\
(i) $\nabla_{fX+gY}^{Z}=f\nabla _{X}^{Z}+g \nabla _{Y}^{Z}$ for all $f,g\in C^{\infty}(M)$ and $X,Y,Z\in \mathfrak{X}(M)$;\\
(ii)  $\nabla_{X}^{aY+bZ}=a\nabla_{X}^{Y}+b\nabla_{X}^{Z}$ for all $a,b\in R$ and $X,Y,Z\in \mathfrak{X}(M)$;\\
(iii) $\nabla_{X}^{fZ}=X(f)Z+ f\nabla_{X}^{Z}$ for all $f\in C^{\infty}(M)$ and $X,Z\in \mathfrak{X}(M)$.\\
If $p\in M$, then the bilinear map $\nabla _{p}:T_{p}M\times \mathfrak{X}(M)\rightarrow T_{p}M$, $(X_{p},Z)\mapsto (\nabla _{X}^{Z})_{p}$ is called the covariant derivative of $Z$ in the direction of $X_{p}$. If $\alpha :(-\epsilon, \epsilon) \rightarrow M$ is a smooth curve and if $Z$ is a smooth vector field along $\alpha $ , then the covariant derivative of $Z$ along $\alpha $ is denoted by $\frac{DZ}{dt}$ and it is defined by $\frac{DZ}{dt}= \nabla _{\dot{\alpha}(t)}^{Y}$, where $Y\in \mathfrak{X}(M)$ is an extension of $Z$ on $M$. If  $\frac{DZ}{dt}=0$ for a vector field $Z$ along $\alpha$, then $Z$  is called a parallel vector field along $\alpha $. If $v\in T_{p}M$ and $\alpha :(-\epsilon, \epsilon) \rightarrow M$ is a smooth curve passing through $p$ that is $\alpha(0)=p$, then it is proved that there is a unique parallel vector field $Z$ along $\alpha $ with $Z_{p}=v$. The mapping $P_{t}:T_{p}M\rightarrow T_{\alpha(t)}M$, $v\mapsto Z_{\alpha (t)}$ is called a parallel transition. A connection $\nabla $ is called a pseudo-Riemannian connection if the parallel transition along any given curve preserves pseudo-Riemannian metric, and it is proved that each pseudo-Riemannian manifold has a unique torsion-free connection \cite{a1, Mu} that is a pseudo-Riemannian connection with the following property:\\ $\nabla_{X}^{Z}-\nabla_{Z}^{X}=[X,Z]$ for all $X,Z\in \mathfrak{X}(M)$.\\ Now we assume that $\nabla$ is a torsion-free pseudo-Riemannian connection on $M$, $p\in M$, and $\alpha:(-\epsilon, \epsilon) \rightarrow M$ is a smooth curve passing through $p$. If $u\in T_{\alpha(t)}M$ and $E$ is a subspace of $T_{\alpha(s)}M$, where $t,s\in (-\epsilon,\epsilon)$,  then we define $d(u,E)$ by $$d(u,E)=inf\{|g_{\alpha(s)}(P_{s-t}(u)- w,P_{s-t}(u)- w)|~:~w\in E~and~|g_{\alpha(s)}(w,w)|=1\}$$ where $P_{t}$ is the parallel transition corresponding to $\nabla .$
\begin{definition}\label{d2}
For two subspaces $E$ of $T_{\alpha(s)}M$ and $F$ of $T_{\alpha(t)}M$ with $s,t\in (-\epsilon,\epsilon)$ we define $d(E,F)$ by $$d(E,F)= max\{a,b\},~where$$ $$a=max\{d(v,F)~:~v\in E~and ~|g_{\alpha(s)}(v,v)|=1\},~and$$ $$b= max\{d(u,E)~:~u\in F~and ~|g_{\alpha(t)}(u,u)|=1\}.$$
\end{definition}
With the former assumptions we have the next theorem.
\begin{theorem}\label{t2}
If $C$ is a hyperbolic set for $f$ up to a distribution $p\mapsto E^{n}(p)$, and if $p\in C$, $\alpha(t_{n})\in C$ and $t_{n}\rightarrow 0$ when $n\rightarrow \infty$, then $$d(E^{s}(\alpha(t_{n})), E^{s}(p))\rightarrow 0,~when ~n\rightarrow \infty$$
and
$$d(E^{u}(\alpha(t_{n})), E^{u}(p))\rightarrow 0,~when ~n\rightarrow \infty.$$
\end{theorem}
\begin{proof} Let $m$ be the dimension of $M$. Then $0\leq dim(E^{s}(\alpha(t_{n}))) \leq m$ for all $n\in N$. Thus there is a subsequence $\{t_{l}~:~l\in N\}$ of $\{t_{n}\}$ and a constant $k\in N$ such that $dim(E^{s}(\alpha(t_{l}))=k$ for all $l\in N$. Let $\{v_{l1}, v_{l2},..., v_{lk}\}$ be a pseudo-orthonormal basis of  $E^{s}(\alpha(t_{l}))$, i.e., $|g_{\alpha(t_{l})}(v_{li}, v_{lj})|=\delta _{ij}$. For a given $i$, the set $\{P_{-t_{l}}(v_{li})~:~ l\in N\}$ is a subset of the compact set  $\{v\in T_{p}M~:~|g_{p}(v,v)|=1\}$. Thus it has a convergent subsequence. We denote it and it's limit  by $\{P_{-t_{r}}(v_{ri})~:~r\in N\}$, and $v_{i}$. Clearly $v_{i}\in E^{s}(p)\oplus E^{u}(p)$, so $v_{i}=u+w$ with $u\in E^{s}(p)$ and $w\in E^{u}(p)$. We have $$a^{-1}b^{-n}|g_{p}(w,w)|\leq |g_{f^{n}(p)}(Df^{n}_{p} (w), Df^{n}_{p} (w))|$$ $$=|g_{f^{n}(p)}(Df^{n}_{p} (w+u-u), Df^{n}_{p} (w+u-u))|$$
 $$\leq |g_{f^{n}(p)}(Df^{n}_{p} (v_{i}), Df^{n}_{p} (v_{i}))| + |g_{f^{n}(p)}(Df^{n}_{p} (u), Df^{n}_{p} (u))|+2|g_{f^{n}(p)}(Df^{n}_{p} (v_{i}),Df^{n}_{p} (u))|$$ $$=\lim_{t_{r}\rightarrow 0} |g_{f^{n}(p)}(Df^{n}_{p} (v_{ri}), Df^{n}_{p} (v_{ri})) |+ |g_{f^{n}(p)}(Df^{n}_{p} (u), Df^{n}_{p} (u))| $$ $$+2\lim_{t_{r}\rightarrow 0} |g_{f^{n}(p)}(Df^{n}_{p} (v_{ri}),Df^{n}_{p} (P_{t_{r}}u))|$$ $$\leq(\lim_{t_{r}\rightarrow 0}a b^{n}g_{p}(v_{ri},v_{ri}))+ a b^{n}g_{p}(u,u)+2\lim_{t_{r}\rightarrow 0} |g_{f^{n}(p)}(Df^{n}_{p} (v_{ri}),Df^{n}_{p} (P_{t_{r}}u))|$$ $$=a b^{n}g_{p}(v_{i},v_{i})+a b^{n}g_{p}(u,u)+2\lim_{t_{r}\rightarrow 0} |g_{f^{n}(p)}(Df^{n}_{p} (v_{ri}),Df^{n}_{p} (P_{t_{r}}u))|.$$ Since $\lim_{n\rightarrow \infty } |g_{f^{n}(p)}(Df^{n}_{p} (v_{ri}),Df^{n}_{p} (P_{t_{r}}u))|=0$,  then the former inequality is valid only in the case $|g_{p}(w,w)|=0$. Thus $w=0$, and $v_{i}\in E^{s}(p)$. The set $\{v_{1}, v_{2},...,v_{k}\}$ is an pseudo-orthonormal subset of $E^{s}(p)$. Thus $dim(E^{s}(p))\geq k$. The similar calculations imply that $dim(E^{u}(p))\geq m-d-k$, where $d$ is the dimension of $E^{n}(p)$. Thus $dim(E^{s}(p))=k$ and $dim(E^{u}(p))=m-d-k$. This implies that for sufficiently large $n$,  $dim(E^{s}(\alpha(t_{n})))=k$ and $dim(E^{u}(\alpha(t_{n})))=m-d-k$.\\
 Now let $\gamma >0$ be given. Since $v_{li}\rightarrow v_{i}$ then  there is $M_{i}>0$ such that for all $l>M_{i}$ $$1-\frac{\gamma}{4}<|g_{p}(P_{-t_{l}}(v_{li}), v_{i})|<1+\frac{\gamma}{4}~  and ~|g_{p}(P_{-t_{l}}(v_{li}), v_{j})|<\frac{\gamma}{4k^{2}}~ for~ all ~j\neq i.$$  If $u_{l}=\displaystyle \Sigma_{i=1}^{k}\beta_{i} v_{li}$ and $|g_{\alpha(t_{l})}(u_{l},u_{l})|=1$, then $$d(u_{l}, E^{s}(p))\leq |g_{p}(P_{-t_{l}}(u_{l})- \displaystyle \Sigma_{i=1}^{k}\beta'_{i} v_{i},P_{-t_{l}}(u_{l})- \displaystyle \Sigma_{i=1}^{k}\beta'_{i} v_{i})|$$
 $$=|g_{p}(\displaystyle \Sigma_{i=1}^{k}\beta_{i}P_{-t_{l}}( v_{li})- \displaystyle \Sigma_{i=1}^{k}\beta'_{i} v_{i},\displaystyle \Sigma_{i=1}^{k}\beta_{i}P_{-t_{l}}( v_{li})- \displaystyle \Sigma_{i=1}^{k}\beta'_{i} v_{i})|$$
 $$= |\displaystyle\Sigma_{i=1}^{k}\Sigma_{j=1}^{k}\beta_{i}\beta'_{j}g_{p}(P_{-t_{l}}( v_{li})-  v_{i},P_{-t_{l}}( v_{lj})-  v_{j})|$$
 $$=2 |\displaystyle\Sigma_{i=1}^{k}(\beta_{i}\beta'_{i})g_{p}(P_{-t_{l}}( v_{li}),  v_{i})|+2 |\displaystyle\Sigma_{i=1}^{k}\displaystyle\Sigma_{i\neq j\& j=1}^{k}(\beta_{i}\beta'_{j})g_{p}(P_{-t_{l}}( v_{li}),  v_{j})|$$ $$ +|\displaystyle\Sigma_{i=1}^{k}(\beta_{i}\beta'_{i})g_{p}(v_{i},  v_{i})|+|\displaystyle\Sigma_{i=1}^{k}(\beta_{i}\beta'_{i})g_{p}(P_{-t_{l}}(v_{li}),  P_{-t_{l}}(v_{li}))|.$$ $$\leq  2|\displaystyle\Sigma_{i=1}^{k}(\beta_{i}\beta'_{i})g_{p}(v_{i},  v_{i})|+ 2 |\displaystyle\Sigma_{i=1}^{k}\displaystyle\Sigma_{i\neq j\& j=1}^{k}(\beta_{i}\beta'_{j})g_{p}(P_{-t_{l}}( v_{li}),  v_{j})|$$ $$ +|\displaystyle\Sigma_{i=1}^{k}(\beta_{i}\beta'_{i})g_{p}(v_{i},  v_{i})|+|\displaystyle\Sigma_{i=1}^{k}(\beta_{i}\beta'_{i})g_{p}(P_{-t_{l}}(v_{li}),  P_{-t_{l}}(v_{li}))|.$$
 If we take $\beta'_{i}=\frac{\gamma\beta_{i}}{4(1+(\Sigma_{i=1}^{k}\beta_{i})^{2})}$, then
  $$d(u_{l}, E^{s}(p))< \frac{\gamma}{4}+\frac{\gamma}{4}+\frac{\gamma}{4}+\frac{\gamma}{4}=\gamma .~(1)$$
  Let $v=\displaystyle\Sigma_{i=1}^{k}\beta_{i}v_{i}\in E^{s}(p)$ and $|g_{p}(v,v)|=1$.  Since $v_{ni}\rightarrow v_{i}$ then there is $L\in N$ such that for all $n>L$
 $$ 1-\frac{\gamma}{4}<|g_{p}(P_{t_{n}}(v_{i}), v_{ni})|<1+\frac{\gamma}{4}~  and ~|g_{p}(P_{t_{n}}(v_{i}), v_{nj})|<\frac{\gamma}{4k^{2}}~ for~ all ~j\neq i.$$
  If $n>L$ then we have $$d(v,E^{s}(\alpha (t_{n})))\leq |g_{\alpha(t_{n})}(P_{t_{n}}(v)-\displaystyle\Sigma_{i=1}^{k}\beta'_{i}v_{ni}, P_{t_{n}}(v)-\displaystyle\Sigma_{i=1}^{k}\beta'_{i}v_{ni})|$$
  $$=|g_{\alpha(t_{n})}(\displaystyle\Sigma_{i=1}^{k}\beta_{i}P_{t_{n}}(v_{i})-\displaystyle\Sigma_{i=1}^{k}\beta'_{i}v_{ni}, \displaystyle\Sigma_{i=1}^{k}\beta_{i}P_{t_{n}}(v_{i})-\displaystyle\Sigma_{i=1}^{k}\beta'_{i}v_{ni})|$$
  $$=|\displaystyle\Sigma_{i=1}^{k}\displaystyle\Sigma_{j=1}^{k}\beta_{i}\beta'_{j}g_{\alpha(t_{n})}(P_{t_{n}}(v_{i})-v_{ni}, P_{t_{n}}(v_{j})-v_{nj})|$$
  $$\leq |\displaystyle\Sigma_{i=1}^{k}2(\beta_{i}\beta'_{i})g_{\alpha(t_{n})}(P_{t_{n}}(v_{i}), v_{ni})|+2 |\displaystyle\Sigma_{i=1}^{k}\displaystyle\Sigma_{i\neq j\& j=1}^{k}(\beta_{i}\beta'_{j})g_{p}(P_{t_{n}}( v_{i}),  v_{nj})|$$ $$+|\displaystyle\Sigma_{i=1}^{k}(\beta_{i}\beta'_{i})g_{\alpha(t_{n})}(v_{ni}, v_{ni})|+|\displaystyle\Sigma_{i=1}^{k}(\beta_{i}\beta'_{i})g_{\alpha(t_{n})}(P_{t_{n}}(v_{i}), P_{t_{n}}(v_{i}))|.$$ If we take $\beta'_{i}=\frac{\gamma\beta_{i}}{4(1+(\Sigma_{i=1}^{k}\beta_{i})^{2})}$, then $$d(v,E^{s}(\alpha (t_{n})))\leq \gamma.~(2)$$
  Inequalities $(1)$ and $(2)$ imply that $$d(E^{s}(\alpha(t_{n})), E^{s}(p))\rightarrow 0,~when ~n\rightarrow \infty.$$
   By the same method we can deduce that $$d(E^{u}(\alpha(t_{n})), E^{u}(p))\rightarrow 0,~when ~n\rightarrow \infty.$$
\end{proof}
\section{Examples}
We consider the  metric  $g_{p}(U,V)=u^{1}v^{1}+u^{2}v^{2}-u^{3}v^{3}-u^{4}v^{4}$ on $R^{4}$, where $p\in R^{4}$ and $U=(u^{1},u^{2},u^{3}, u^{4}),~ V=(v^{1},v^{2},v^{3},v^{4})\in T_{p}R^{4}$.
\begin{example}\label{e00}
If we define $f:R^{4}\rightarrow R^{4}$ by $f(x,y,z,t)=(\frac{x}{2}, \frac{y}{3}, z, 4t)$ then $C=\{(0,0,0,0)\}$ is a hyperbolic set for $f$ up to the distribution $p\mapsto E^{n}(p)=\{(a,a,a,a)~:~a\in R\}$, but $C$ is not a hyperbolic set  for $f$ up to the distribution $p\mapsto E^{n}(p)=\{(a,a,0,a\sqrt{2})~:~a\in R\}.$
\end{example}
The Lorentz or Minkowski metric on $R^{3}$ is defined by $g_{p}(U,V)=u^{1}v^{1}+u^{2}v^{2}-u^{3}v^{3}$ where $p\in R^{3}$ and $U=(u^{1},u^{2},u^{3}),~ V=(v^{1},v^{2},v^{3})\in T_{p}R^{3}$.
 \begin{example}\label{e0}
  Let $\Lambda$ be the hyperbolic set of the Smale horseshoe $h:R^{2}\rightarrow R^{2}$ \cite{S}. If we define $\tilde{h}:R^{3}\rightarrow R^{3}$ by $\tilde{h}(x,y,z)=(h(x,y),4z)$, then $\Lambda \times R$ is a hyperbolic set for  $\tilde{h}$ up to the distribution $p\mapsto E^{n}(p)=\{(a,a,a\sqrt{2})~:~a\in R\}$, where the metric is the Lorentz metric.
 \end{example}
 The manifold $H^{2}(1)=\{(x,y,z)\in R^{3}~|~x^{2}+y^{2}-z^{2}=-1~and~z>0\}$ with the induced Minkowski metric is a Riemannian manifold which is called a hyperbolic space (see figure 1).
\begin{figure}[]
\centering
\includegraphics[width=8cm, height=6cm]{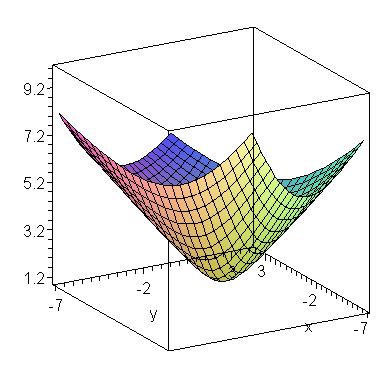}
\caption{The hyperbolic space $H^{2}(1)$.}
\label{}
\end{figure}
\begin{example}\label{e1}
The mapping $f:H^{2}(1)\rightarrow H^{2}(1)$ defined by $$f(x,y,z)=(\frac{x}{2}, \frac{y}{2}, \sqrt{\frac{z^{2}}{4}+ \frac{3}{4}} )$$ is a smooth diffeomorphism. The set $C=\{(0,0,1)\}$ is a hyperbolic set for $f$. Because if $p=(0,0,1)$ and $V=(u,v,w)\in T_{p}H^{2}(1)$, then there is a smooth curve $\beta :(-\epsilon,\epsilon)\rightarrow H^{2}(1)$ such that $\beta(0)=p$ and $\dot{\beta}(0)=V$. We have $(\beta_{1}(t))^{2}+(\beta_{2}(t))^{2}-(\beta_{3}(t))^{2}=-1$. Thus $\beta_{1}(t)\dot{\beta}_{1}(t)+\beta_{2}(t)\dot{\beta}_{2}(t)-\beta_{3}(t)\dot{\beta}_{3}(t)=0$. Hence $0u+0v-w=0$, so $w=0$. By computing the derivative of $f^{n}o\beta$ at zero we find $D_{p}f^{n}(V)=(\frac{u}{2^{n}},\frac{v}{2^{n}},0)$, where $n\in N$. Hence $$g_{p}(D_{p}f^{n}(V), D_{p}f^{n}(V))=\frac{u^{2}}{2^{2n}}+\frac{v^{2}}{2^{2n}}= \frac{1}{2^{2n}}g_{p}(V,V).$$
If we take $E^{s}(p)=T_{p}H^{2}(1)$, $E^{u}(p)=E^{n}(p)=\{(0,0,0)\}$, $a=1$ and $b=\frac{1}{4}$ then we have the conditions of definition \ref{d1}.
\end{example}
In the former example $\{(0,0,1)\}$ is a global attractor for $f$ i.e, the $\omega$-limits of all the points of the space is $\{(0,0,1)\}$. In the next example we present a global attractor which is not a hyperbolic set.
\begin{example}\label{e2}
We define $h:M=H^{2}(1)-\{(0,0,1)\}\rightarrow M$ by $$h(x,y,z)=((\sqrt{\frac{z^{2}+2\sqrt{2}z-2}{4z^{2}-4}}x, \sqrt{\frac{z^{2}+2\sqrt{2}z-2}{4z^{2}-4}}y, \frac{z+\sqrt{2}}{2}).$$
$C=\{(x,y,\sqrt{2})~|~x^{2}+y^{2}=1\} $ is a compact invariant set for $h$ and it is a global attractor for it, but it is not a hyperbolic set for $h$ (see figure 2). In fact if $p=(x,y,\sqrt{2})\in C$, and $V=(u,v,w)\in T_{p}M$, then $xu-yv-\sqrt{2}w=0$, Thus $V=(u,v, \frac{\sqrt{2}}{2}xu+ \frac{\sqrt{2}}{2}yv)$.  We have $$D_{p}h(V)=(-\frac{ux^{2}}{2}-\frac{vyx}{2}+u, -\frac{uxy}{2}-\frac{vy^{2}}{2}+v, \frac{\sqrt{2}}{4}(ux+vy)).$$
 Hence $g_{p}(D_{p}h(V), D_{p}h(V))=g_{p}(V,V)-\frac{5}{16}(xu+yv)^{2}$. Thus the derivative of $h$ is a Riemannian metric preserving on the set of vectors of the form $(-\frac{y}{x}v,v,0)$ when $x\neq 0$ and $(u, -\frac{x}{y}u,0)$ when $y\neq 0$, and it is a decreasing map (not exponentially) on the rest of tangent space at $p$. This implies that $C$ is not a hyperbolic set for $h$.
\begin{figure}[]
\centering
\includegraphics[width=8cm, height=6cm]{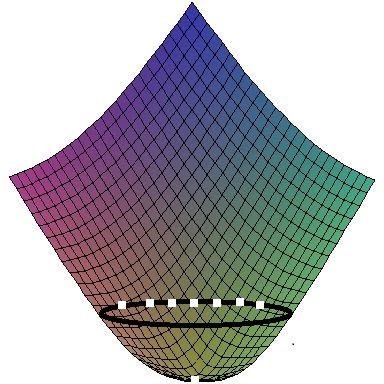}
\caption{The black circle $C$ is a global attractor for $h$ but it is not a hyperbolic set for it.}
\label{}
\end{figure}
\end{example}
\section{Conclusion}
We present a definition for hyperbolic dynamics on a pseudo-Riemannian manifold. We prove the uniqueness of stable and unstable subspaces up to a distribution, and we prove the continuity of this decomposition. By an example we present a global attractor which is not hyperbolic.

\end{document}